\documentclass[preprint,12pt]{elsarticle}

\usepackage{amsmath, amssymb, amsthm, mathtools}
\usepackage{enumitem}
\setlist[enumerate]{itemsep=2pt, topsep=4pt}
\usepackage{graphicx}
\usepackage{float}
\usepackage{xcolor}
\usepackage{tikz}
\usetikzlibrary{arrows.meta, positioning, calc}
\usepackage[margin=2.5cm]{geometry}
\usepackage{placeins}
\usepackage{microtype}
\usepackage{hyperref}

\theoremstyle{plain}
\newtheorem{theorem}{Theorem}[section]
\newtheorem{lemma}[theorem]{Lemma}
\newtheorem{proposition}[theorem]{Proposition}
\theoremstyle{definition}
\newtheorem{construction}[theorem]{Construction}

\numberwithin{equation}{section}
\numberwithin{figure}{section}

\newcommand{\one}{\mathbf{1}}

\newcommand{\Q}{\mathbb{Q}}
\newcommand{\R}{\mathbb{R}}
\newcommand{\Z}{\mathbb{Z}}
\newcommand{\calK}{\mathcal{K}}
\newcommand{\calG}{\mathcal{G}}
\newcommand{\sgn}{\sigma}
\newcommand{\Span}{\operatorname{span}}
\DeclareMathOperator{\rank}{rank}

\begin{document}

\begin{frontmatter}

\title{Graphs with zero as a main eigenvalue of the signless Laplacian}

\author[1]{Hangxi Cha}
\ead{2433949@tongji.edu.cn}
\author[1]{Haiying Shan\corref{cor1}}
\ead{shan_haiying@tongji.edu.cn}
\affiliation[1]{organization={School of Mathematical Sciences, Tongji University},
            city={Shanghai},
            postcode={200092},
            country={P.R. China}}
\cortext[cor1]{Corresponding author}

\begin{abstract}
An eigenvalue of the signless Laplacian \(Q(G)\) is \(Q\)-main if its
eigenspace is not orthogonal to the all-ones vector.  We characterize graphs
with exactly \(\ell\ge3\) \(Q\)-main eigenvalues, one of which is zero.  The
case \(\ell=3\) reduces to non-semiregular bipartite graphs satisfying a
vertexwise signed degree-sum identity.  For each integer \(k\ge0\), we
construct infinitely many pairwise nonisomorphic graphs of cyclomatic number
\(k\) and unbounded diameter, all with exactly three \(Q\)-main eigenvalues
including zero.  These families provide counterexamples to the stated
classifications of trees, unicyclic graphs, and bicyclic graphs of Javarsineh
and Fath-Tabar.
\end{abstract}


\end{frontmatter}

\section{Introduction}

Throughout the paper, graphs are finite, simple, connected, undirected, and
have at least one edge.  For a graph \(G\), let \(A(G)\) and \(D(G)\) denote
its adjacency and degree matrices.  The signless Laplacian is
\(Q(G)=A(G)+D(G)\).  The cyclomatic number of \(G\) is
\(|E(G)|-|V(G)|+1\).
An eigenvalue \(\lambda\) of \(Q(G)\) is \emph{\(Q\)-main} if the
\(\lambda\)-eigenspace is not orthogonal to the all-ones vector \(\one\).
We study graphs with exactly three \(Q\)-main eigenvalues, including the zero
eigenvalue.

Main eigenvalues for the adjacency matrix were introduced by
Cvetkovi\'c~\cite{Cvetkovic1970}.  Hagos proved that the number of main
eigenvalues equals the rank of the walk matrix~\cite{Hagos2002}; see also
Rowlinson's survey~\cite{Rowlinson2007}.  For the signless Laplacian, the
multiplicity of the zero eigenvalue equals the number of bipartite components
of the graph \cite{Cvetkovic2007,CvetkovicSimic2009,CvetkovicBook}; for the
connected graphs considered here it is \(0\) or \(1\).

Several structural results are known for \(Q\)-main eigenvalues.  Deng and
Huang proved that a graph has exactly one \(Q\)-main eigenvalue if and only if
it is regular~\cite{DengHuang2013}.  Chen and Huang proved that, for a graph
of order \(n\), the number of \(Q\)-main eigenvalues equals
\(\rank[\one,Q\one,\ldots,Q^{n-1}\one]\), and classified trees, unicyclic
graphs, and bicyclic graphs with exactly two \(Q\)-main
eigenvalues~\cite{ChenHuang2013}.  Li and Yang characterized the tricyclic
case with exactly two \(Q\)-main eigenvalues~\cite{LiYang2013Tricyclic}.
Vinagre, Trevisan, Bolckau, and Chimelli characterized threshold graphs with a
prescribed number of \(Q\)-main eigenvalues, and Jones, Trevisan, and Vinagre
characterized quasi-threshold graphs with two \(Q\)-main eigenvalues
\cite{VinagreTrevisanBolckauChimelli2020,JonesTrevisanVinagre2025}.

The zero eigenvalue is restrictive: it can be \(Q\)-main only for bipartite
graphs whose two bipartition classes have unequal sizes.  After making this
reduction, we give necessary and sufficient conditions for a graph to have
exactly \(\ell\) (\(\ell\ge3\)) \(Q\)-main eigenvalues, including zero.  For
\(\ell=3\), the criterion becomes a signed degree-sum equation together with
non-semiregularity.  We then construct, for every \(k\ge0\), infinitely many
pairwise nonisomorphic graphs with cyclomatic number \(k\), unbounded
diameter, and prescribed \(Q\)-main polynomial.

Applying this construction with \(k=0,1,2\) gives trees, unicyclic graphs, and
bicyclic graphs with exactly three \(Q\)-main eigenvalues, including the zero
eigenvalue, that are absent from the stated classifications of Javarsineh and
Fath-Tabar
\cite{JavarsinehFathTabar2017Filomat,JavarsinehFathTabar2017AMC}.  Hence those
classification theorems are not correct as complete classifications.

The paper is organized as follows.  Section~\ref{sec:preliminaries} collects
the linear-algebraic and signless-Laplacian preliminaries.
Section~\ref{sec:zero-q-main} gives the criterion after the reduction to
bipartite graphs with unequal classes, and derives its specialization to
exactly three \(Q\)-main eigenvalues.
Section~\ref{sec:families} gives the construction and its consequences for
cyclomatic numbers \(0,1,2\).  Section~\ref{sec:open-problems} discusses
possible extensions and open problems.

\section{Preliminaries}\label{sec:preliminaries}

Let \(M\) be a real symmetric matrix and \(z\) a nonzero vector.  An eigenvalue
\(\lambda\) of \(M\) is \emph{main with respect to \(z\)} (or \(z\)-main) if
the orthogonal projection of \(z\) onto \(\ker(M-\lambda I)\) is nonzero.
The Krylov space generated by \(z\) is
\(\calK(M, z)=\Span\{z, Mz, M^2z, \ldots\}\).  The minimal polynomial of \(M\) on
\(\calK(M, z)\) is the monic polynomial \(m_{M,z}\) of least degree such that
\(m_{M,z}(M)z=0\); equivalently, it is the minimal polynomial of the restriction
of \(M\) to the invariant subspace \(\calK(M, z)\).  For \(M=Q(G)\) and
\(z=\one\), the \(z\)-main eigenvalues are precisely the \(Q\)-main eigenvalues
of \(G\).  If these distinct eigenvalues are
\(\lambda_1,\ldots,\lambda_\ell\), set
\(\Phi_Q(G,x)=\prod_{i=1}^\ell(x-\lambda_i)\).
We call \(\Phi_Q(G,x)\) the \(Q\)-main polynomial of \(G\).  The next lemma
records standard facts about cyclic subspaces.  A proof is included for
completeness.  In particular, since
\(\calK(Q,\one)=\Span\{\one,Q\one,\ldots,Q^{n-1}\one\}\), part~(i) recovers the
walk-matrix rank formula of Chen and Huang: the number of distinct \(Q\)-main
eigenvalues equals \(\rank[\one,Q\one,\ldots,Q^{n-1}\one]\)~\cite{ChenHuang2013}.

\begin{lemma}[{\cite[Proposition~1]{GutknechtSchmelzer2009} for
parts~\textup{(i)}--\textup{(ii)}}] \label{lem:krylov}
Let \(M\) be a real symmetric matrix and let \(z\ne0\).  Suppose that \(M\) has
exactly \(\ell\) distinct \(z\)-main eigenvalues
\(\lambda_1,\ldots,\lambda_\ell\).
\begin{enumerate}[label=(\roman*)]
    \item \(\dim\calK(M,z)=\ell\) and
    \(m_{M,z}(x)=\prod_{i=1}^\ell (x-\lambda_i)\).
    \item A monic polynomial \(f\) satisfies \(f(M)z=0\) if and only if
    \(m_{M,z}\mid f\).
    \item For \(M=Q(G)\) and \(z=\one\), \(m_{Q(G),\one}(x)=\Phi_Q(G,x)\).
    \item If \(M\in\Z^{n\times n}\) and \(z\in\Z^n\), then \(m_{M,z}\in\Z[x]\).
\end{enumerate}
\end{lemma}

\begin{proof}
Let \(E_\lambda\) denote the orthogonal projection onto the
\(\lambda\)-eigenspace of \(M\).  The spectral decomposition gives
\[
        M^jz=\sum_{\lambda} \lambda^j E_\lambda z .
\]
Hence \(\calK(M,z)\) is spanned by the nonzero vectors
\(E_{\lambda_i}z\).  The nonsingularity of the corresponding Vandermonde
matrix shows that these vectors already lie in
\[
        \Span\{z,Mz,\ldots,M^{\ell-1}z\}.
\]
Thus \(\dim\calK(M,z)=\ell\), and the restriction of \(M\) to this space has
distinct eigenvalues \(\lambda_1,\ldots,\lambda_\ell\); its minimal polynomial
is therefore \(\prod_{i=1}^\ell(x-\lambda_i)\).  This proves~(i), and~(ii) is
the usual divisibility property of the minimal polynomial.  Statement~(iii) is
the special case \(M=Q(G)\), \(z=\one\), since \(\Phi_Q(G,x)\) is by definition
the product \(\prod_{i=1}^\ell(x-\lambda_i)\) over the \(Q\)-main eigenvalues.

It remains to prove~(iv).  Suppose that \(M\) and \(z\) are integral.  The
vectors \(z,Mz,\ldots,M^{\ell-1}z\) are then linearly independent over \(\Q\),
while \(M^{\ell}z\) is an \(\R\)-linear, hence a \(\Q\)-linear, combination of
them; so \(m_{M,z}\in\Q[x]\).  As a monic divisor of the monic integral
characteristic polynomial of \(M\), it lies in \(\Z[x]\) by Gauss's lemma.
\end{proof}

For a bipartite graph \(G\), we write an ordered bipartition as
\(V(G)=V_+\sqcup V_-\).  Since \(G\) is connected, this bipartition is unique
up to swapping the two classes; when they have unequal cardinalities and the
construction at hand prescribes no order, we label them so that
\(|V_+|>|V_-|\).  The inequality \(|V_+|>|V_-|\) is thus a labelling
convention, not a condition to be checked.  The sign
vector of this ordered bipartition is the vector
\(\sgn\in\{1,-1\}^{V(G)}\) given by
\[
        \sgn(v)=
        \begin{cases}
       1, & v\in V_+,\\
      -1, & v\in V_-.
        \end{cases}
\]

\begin{lemma} \label{lem:zero-main}
Let \(G\) be a graph.  Then \(0\) is a \(Q\)-eigenvalue of \(G\) if and only if
\(G\) is bipartite.  If \(G\) is bipartite, then for an ordered bipartition
\(V(G)=V_+\sqcup V_-\) with sign vector \(\sgn\),
\(\ker Q(G)=\operatorname{span}\{\sgn\}\), and the zero eigenvalue is \(Q\)-main
if and only if \(|V_+|\ne |V_-|\).
\end{lemma}

\begin{proof}
Let \(B\) be the \(0\)-\(1\) vertex-edge incidence matrix of \(G\), so that
\(Q(G)=BB^T\).  Recall that \(G\) is connected; for a connected graph on
\(n\) vertices, \(\rank B=n-1\) if \(G\) is bipartite and \(\rank B=n\)
otherwise \cite[Lemma~2.17]{Bapat2010}.  Since \(\ker Q(G)=\ker B^T\), the
value \(0\) is a \(Q\)-eigenvalue precisely in the bipartite case, and its
multiplicity is then one.  When \(G\) is bipartite, \(B^T\sgn=0\), so this
one-dimensional kernel is \(\operatorname{span}\{\sgn\}\).  The projection of
\(\one\) onto \(\ker Q\) is nonzero exactly when
\(\langle \one,\sgn\rangle=|V_+|-|V_-|\ne 0\).
\end{proof}

For a graph \(G\), write \(A=A(G)\) and put \(d=A\one\) and \(s=Ad=A^2\one\),
so that \(d(v)=\deg_G(v)\) and \(s(v)=\sum_{u\sim v}d(u)\); the latter is the
\emph{2-degree} of \(v\).
If \(Q=Q(G)\), then
\begin{equation}\label{eq:degree-sum-identities}
        (Q\one)_v=2d(v),\qquad
        (Q^2\one)_v=2(Qd)_v=2d(v)^2+2s(v).
\end{equation}
These identities turn a polynomial relation in \(Q\) applied to \(\one\) into
pointwise conditions on the degrees and 2-degrees, and underlie both the
criterion of Section~\ref{sec:zero-q-main} and the constructions of
Section~\ref{sec:families}.

Let \(a,b\in\Z\) and \(t\in\Z_{>0}\).  For a bipartite graph \(G\) with
ordered bipartition \(V(G)=V_+\sqcup V_-\) and sign vector \(\sgn\), the
\emph{\((a,b,t)\)-signed degree-sum equation} is
\[
        Q^2\one - aQ\one + b\one = t\sgn .
\]
By \eqref{eq:degree-sum-identities}, this equation is equivalent to
\[
        s(v)=
        \begin{cases}
        -d(v)^2+ad(v)+(t-b)/2, & v\in V_+,\\
        -d(v)^2+ad(v)-(t+b)/2, & v\in V_-.
        \end{cases}
\]

\section{A criterion for the zero \texorpdfstring{\(Q\)-main}{Q-main} case}
\label{sec:zero-q-main}

By Lemma~\ref{lem:zero-main}, the zero eigenvalue is \(Q\)-main precisely when
the graph is bipartite and the two bipartition classes have unequal
cardinalities.  We therefore work directly in this setting and use the
ordering convention from Section~\ref{sec:preliminaries}; thus
\(V(G)=V_+\sqcup V_-\) and \(\sgn\) denotes the corresponding sign vector.
The following criterion combines the cyclic-subspace description of main
eigenvalues with the one-dimensional kernel of \(Q\).

\begin{theorem} \label{thm:zero-main-general-criterion}
Let \(\ell\ge 3\), and let \(G\) be a bipartite graph whose two bipartition
classes have unequal cardinalities.  Let \(V(G)=V_+\sqcup V_-\) and \(\sgn\)
be as in the convention above.  The following are equivalent.
\begin{enumerate}[label=(\alph*)]
    \item \(G\) has exactly \(\ell\) \(Q\)-main eigenvalues, one of which is
    zero.
    \item There exist \(t\in\Z_{>0}\) and a monic polynomial \(p\in\Z[x]\) of
    degree \(\ell-1\), whose roots are all distinct and positive, such that
    \begin{enumerate}[label=(\roman*)]
        \item \(\one,Q\one,\ldots,Q^{\ell-2}\one,\sgn\) are linearly
        independent;
        \item \(p(Q)\one=(-1)^{\ell-1}t\sgn\).
    \end{enumerate}
\end{enumerate}
In this case
\[
        \Phi_Q(G,x)=xp(x).
\]
\end{theorem}

\begin{proof}
\emph{Necessity.} Suppose that (a) holds, and let the \(Q\)-main eigenvalues be
\(0,\theta_1,\ldots,\theta_{\ell-1}\).  These are distinct, and since \(Q\) is
positive semidefinite, the nonzero ones \(\theta_i\) are positive.  Put
\(p(x)=\prod_{i=1}^{\ell-1}(x-\theta_i)\), a monic polynomial of degree
\(\ell-1\) with distinct positive roots; then \(\Phi_Q(G,x)=xp(x)\).  Since
\(Q\) and \(\one\) are integral, Lemma~\ref{lem:krylov} shows that \(xp(x)\)
has integer coefficients, and as its constant term vanishes, \(p\in\Z[x]\).

By Lemma~\ref{lem:krylov}, \(Qp(Q)\one=0\), so \(p(Q)\one\in\ker Q\).  This
vector is nonzero, for otherwise the minimal polynomial of \(Q\) on
\(\calK(Q,\one)\), which has degree \(\ell\), would divide the polynomial
\(p\) of degree \(\ell-1\).  Hence Lemma~\ref{lem:zero-main} gives
\(p(Q)\one=c\sgn\) for some nonzero scalar \(c\).

With the chosen order, \(\langle\one,\sgn\rangle>0\).  Since \(Q\sgn=0\),
\[
        c|V(G)|
        =\langle p(Q)\one,\sgn\rangle
        =\langle\one,p(Q)\sgn\rangle
        =p(0)\langle\one,\sgn\rangle .
\]
The roots of \(p\) are positive, so \(p(0)\) has sign \((-1)^{\ell-1}\).
Thus \(p(Q)\one=(-1)^{\ell-1}t\sgn\) for some \(t\in\R_{>0}\).  Since
\(p\in\Z[x]\), the vector \(p(Q)\one\) is integral, and hence
\(t\in\Z_{>0}\).

It remains to prove (b)(i).  Since \(\dim\calK(Q,\one)=\ell\), the vectors
\(\one,Q\one,\ldots,Q^{\ell-1}\one\) are linearly independent.  Writing
\(q(x)=p(x)-x^{\ell-1}\), so that
\(\deg q\le \ell-2\), we have
\[
        \sgn=\frac{(-1)^{\ell-1}}{t}\left(
        Q^{\ell-1}\one+q(Q)\one
        \right).
\]
The coefficient of \(Q^{\ell-1}\one\) in this expression is nonzero; replacing
\(Q^{\ell-1}\one\) by \(\sgn\) therefore preserves linear independence.
This proves (b)(i).

\emph{Sufficiency.} Suppose that (b) holds.  Since \(Q\sgn=0\),
\(Qp(Q)\one=(-1)^{\ell-1}tQ\sgn=0\).
Let \(\mu(x)\) be the minimal polynomial of \(Q\) on \(\calK(Q,\one)\).  By
Lemma~\ref{lem:krylov}, \(\mu(x)\) divides \(xp(x)\), so \(\deg\mu\le \ell\).
Furthermore, \(\sgn=\frac{(-1)^{\ell-1}}{t} p(Q)\one\in\calK(Q,\one)\).
By (b)(i), \(\calK(Q,\one)\) contains the \(\ell\) linearly independent
vectors \(\one,Q\one,\ldots,Q^{\ell-2}\one,\sgn\), whence
\(\dim\calK(Q,\one)\ge \ell\).  Thus \(\deg\mu=\dim\calK(Q,\one)\ge \ell\),
and consequently \(\deg\mu=\ell\).  Since \(\mu(x)\) and
\(xp(x)\) are monic of the same degree and \(\mu(x)\mid xp(x)\), we have
\(\mu(x)=xp(x)\).  Lemma~\ref{lem:krylov} gives
\(\Phi_Q(G,x)=xp(x)\).  Thus \(G\) has exactly \(\ell\) \(Q\)-main eigenvalues,
and the zero eigenvalue is \(Q\)-main by Lemma~\ref{lem:zero-main}.
\end{proof}

For \(\ell=3\), the linear independence condition in
Theorem~\ref{thm:zero-main-general-criterion} has a simple degree-theoretic
form.  A bipartite graph with bipartition \(V_+\sqcup V_-\) is
\emph{semiregular} if its degree is constant on each bipartition class.
For such a graph, if these two constants are \(r_+\) and \(r_-\), then
\[
        d=\frac{r_++r_-}{2}\one+\frac{r_+-r_-}{2}\sgn,
\]
so \(\one,d,\sgn\) are linearly dependent.  Conversely, any nontrivial
relation \(\alpha\one+\beta d+\gamma\sgn=0\) has \(\beta\ne0\); otherwise
\(\sgn\) would be a scalar multiple of \(\one\).  Hence
\(d(v)=-(\alpha+\gamma\sgn(v))/\beta\), which is constant on each bipartition
class.  Thus semiregularity is equivalent to the linear dependence of
\(\one,d,\sgn\).
The criterion of Deng and Huang for two \(Q\)-main
eigenvalues~\cite[Theorem~2.5]{DengHuang2013} and the equality characterization
of Feng and Yu~\cite[Lemma~2.3]{FengYu2009} show that a nontrivial connected
graph \(G\) is a nonregular semiregular bipartite graph if and only if its
\(Q\)-main eigenvalues are precisely \(0\) and its largest \(Q\)-eigenvalue.

\begin{theorem} \label{thm:three-main-criterion}
Let \(G\) be a bipartite graph whose two bipartition classes have unequal
cardinalities.  Let \(V(G)=V_+\sqcup V_-\) and \(\sgn\) be as in the convention
above.  Then \(G\) has exactly three \(Q\)-main eigenvalues, one of which is
zero, if and only if there exist \(a,b,t\in\Z_{>0}\) with \(a^2>4b\) such that
\begin{enumerate}[label=(\roman*)]
    \item \(G\) is not semiregular;
    \item \(Q^2\one - aQ\one + b\one = t\sgn\).
\end{enumerate}
In this case the \(Q\)-main polynomial is \(x(x^2 - ax + b)\).  Moreover, any
such triple \((a,b,t)\) necessarily satisfies \(b\equiv t\pmod 2\) and
\(0<t<b\).
\end{theorem}

\begin{proof}
Apply Theorem~\ref{thm:zero-main-general-criterion} with \(\ell=3\).  Write
\(p(x)=x^2-ax+b\).  For a monic quadratic with integral coefficients, having
two distinct positive roots is equivalent to
\(a,b\in\Z_{>0}\) and \(a^2>4b\).  The linear-independence condition in
Theorem~\ref{thm:zero-main-general-criterion} is that
\(\one,Q\one,\sgn\) be linearly independent.  Since \(Q\one=2d\), this is
equivalent to the linear independence of \(\one,d,\sgn\), which is precisely
non-semiregularity by the preceding paragraph.

It remains to record the elementary restrictions on \((a,b,t)\).  Evaluating
the signed degree-sum equation at a vertex gives
\(2d(v)^2+2s(v)-2ad(v)+b=t\sgn(v)\), whence \(b\equiv t\pmod 2\).
Taking its inner product with \(\sgn\) gives
\(t|V(G)|=b(|V_+|-|V_-|)\),
because \(Q\sgn=0\).  By the ordering convention, \(|V_+|>|V_-|\); since both
bipartition classes are nonempty, \(0<t<b\).
\end{proof}

\section{Infinite families of prescribed cyclomatic number}\label{sec:families}

Theorem~\ref{thm:three-main-criterion} reduces the property of having exactly
three \(Q\)-main eigenvalues, one of which is zero, to a single vertexwise
identity.  We now show that its solutions are far from sporadic: for every
cyclomatic number they form infinite families of unbounded diameter.

For \(a,b\in\Z\) and \(t\in\Z_{>0}\), let \(\calG(a,b,t)\) denote the class of
bipartite graphs with unequal bipartition classes, ordered as in
Section~\ref{sec:preliminaries}, that have exactly three \(Q\)-main
eigenvalues including zero and satisfy the \((a,b,t)\)-signed degree-sum
equation.  Let \(\calG_k(a,b,t)\) be the subclass of cyclomatic number \(k\).

Theorem~\ref{thm:three-main-criterion} gives the necessary conditions
\(b\equiv t\pmod2\) and \(0<t<b\).  Thus
\[
        b=t+2q \qquad (q\ge1).
\]
The parameters used below come from the following calculation.  Suppose that
the vertices on the \(W_-\)-side have degree \(t+q\), and that all their
neighbours in the final graph have degree \(t^2+1\).  The signed degree-sum
equation at such a vertex is
\[
        (t+q)(t^2+1)=-(t+q)^2+a(t+q)-(t+q),
\]
so
\[
        a=t^2+t+q+2.
\]
Hence, for \(t=r\) and \(b=r+2q\), this degree arrangement leads to
\((a,b,t)=(r^2+r+q+2,r+2q,r)\).  We first record examples on the line
\(t=1\), and then realize the case \(q=1\), \(r\ge2\), with arbitrary
prescribed cyclomatic number.

The boundary case \(q=1\), \(r=1\), namely \((a,b,t)=(5,3,1)\), is also
realized, although it is not needed in the infinite-family construction below.
Take the path \(v_1u_1vu_2v_2\), where \(u_1,u_2\in V_+\) and
\(v,v_1,v_2\in V_-\), and add two leaves in \(V_+\) adjacent to each of
\(v_1\) and \(v_2\).  The resulting tree satisfies
\[
        s(v)=-d(v)^2+5d(v)-1 \quad (v\in V_+),
        \qquad
        s(v)=-d(v)^2+5d(v)-2 \quad (v\in V_-),
\]
and hence belongs to \(\calG_0(5,3,1)\).

The same line \(t=1\) also has infinite families in all higher cases except
possibly \(q=2\).

\begin{proposition}\label{prop:t-one-family}
For every \(q\ge3\), there are infinitely many connected graphs in
\(\calG(q+4,2q+1,1)\).
\end{proposition}

\begin{proof}
Let \(n\ge1\) and put \(N=2n(q+1)\).  Start with a connected \(q\)-regular graph
on \(N\) vertices.  Subdivide every edge once, put the original vertices in
\(V_+\), and put the subdivision vertices in \(V_-\).  Next, for each
subdivision vertex, add a new vertex in \(V_+\) adjacent only to it.  Since the
initial graph has \(nq(q+1)\) edges, this gives \(nq(q+1)\) new vertices in
\(V_+\).

Partition these new \(V_+\)-vertices into \(nq\) sets of size \(q+1\).  For each
set, add one new vertex in \(V_-\) and join it to all vertices of that set.
Let \(G\) be the resulting bipartite graph.  Then \(G\) is connected.

The original \(V_+\)-vertices have degree \(q\), and all their neighbours have
degree \(3\).  The new \(V_+\)-vertices have degree \(2\), with one neighbour of
degree \(3\) and one neighbour of degree \(q+1\).  The subdivision vertices in
\(V_-\) have degree \(3\), with neighbours of degrees \(q,q,2\), and the further
vertices in \(V_-\) have degree \(q+1\), all of whose neighbours have degree
\(2\).  Therefore
\[
        s(v)=-d(v)^2+(q+4)d(v)-q \quad (v\in V_+),
        \qquad
        s(v)=-d(v)^2+(q+4)d(v)-(q+1) \quad (v\in V_-),
\]
which is the \((q+4,2q+1,1)\)-signed degree-sum equation.  Moreover,
\[
        |V_+|=n(q+1)(q+2),\qquad |V_-|=nq(q+2),
\]
so the prescribed \(V_+\)-class is the larger one.  Since \(q\ge3\), the graph
is not semiregular, and \((q+4)^2>4(2q+1)\).  Theorem~\ref{thm:three-main-criterion}
therefore gives \(G\in\calG(q+4,2q+1,1)\).
\end{proof}

Such a connected \(q\)-regular graph on \(N=2n(q+1)\) vertices may be chosen on
the vertex set of the cycle \(C_N\) as follows.  If \(q=2s\), join every two
vertices whose distance on \(C_N\) is at most \(s\).  If \(q=2s+1\), add to this
graph the perfect matching joining opposite vertices of \(C_N\).

We now turn to the first layer \(q=1\) with \(t=r\ge2\).  The next construction
extends a suitable auxiliary bipartite graph by attaching stars and double
stars.

\begin{construction} \label{con:F}
Fix \(r\ge2\).  Let \(H\) be a bipartite graph with bipartition
\(W_+\sqcup W_-\) such that every vertex in \(W_-\) has degree \(r+1\),
and every vertex in \(W_+\) has degree either \(1\) or \(r^2\).

Obtain \(F_r(H)\) from \(H\) as follows.  For distinct choices of \(w\) and
\(j\), all stars and double stars used below are taken pairwise disjoint before
the indicated identifications.
\begin{enumerate}[label=(\roman*)]
    \item For each vertex \(w\in W_+\) with \(d_H(w)=r^2\), identify \(w\)
    with a leaf of a new copy of \(K_{1,r^2+r+1}\).
    \item For each vertex \(w\in W_+\) with \(d_H(w)=1\), take \(r\) new
    double stars \(D_{w,1},\ldots,D_{w,r}\).  In \(D_{w,j}\), let \(c_{w,j}\)
    and \(c'_{w,j}\) be the two centers, with degrees \(r\) and \(r^2+r+1\)
    in \(D_{w,j}\), respectively.  Identify \(w\) with \(c_{w,j}\) for every
    \(j\).
\end{enumerate}
\end{construction}

\begin{lemma} \label{lem:extension}
For \(H\) as in Construction~\ref{con:F}, let \(V_+\sqcup V_-\) be the
bipartition of \(F_r(H)\) extending the prescribed bipartition
\(W_+\sqcup W_-\) of \(H\).  Then this prescribed \(V_+\)-class is the larger
one, \(F_r(H)\) has the same cyclomatic number as \(H\), and \(F_r(H)\)
satisfies the \((r^2+r+3,r+2,r)\)-signed degree-sum equation.
\end{lemma}

\begin{proof}
Each attached graph is a tree, and in each case exactly one of its vertices is
identified with a vertex of \(H\).  No edge is added between vertices already in
\(H\).  Hence bipartiteness is preserved and the cyclomatic number is unchanged.

Let
\(X=\{w\in W_+:d_H(w)=r^2\}\) and
\(Y=\{w\in W_+:d_H(w)=1\}\).  By construction,
\(d_{F_r(H)}(w)=r^2+1\) for \(w\in W_+\), and
\(d_{F_r(H)}(u)=r+1\) for \(u\in W_-\).  Moreover,
\(V(F_r(H))\setminus V(H)\) consists precisely of the following vertices.  If
\(w\in X\), the copy of \(K_{1,r^2+r+1}\) attached at \(w\) has its center in
\(V_-\), of degree \(r^2+r+1\), and its remaining \(r^2+r\) leaves in \(V_+\).
If \(w\in Y\) and \(1\le j\le r\), then
\(c'_{w,j}\in V_-\) has degree \(r^2+r+1\); the \(r-1\) leaves adjacent to
\(c_{w,j}=w\) lie in \(V_-\), and the \(r^2+r\) leaves adjacent to
\(c'_{w,j}\) lie in \(V_+\).

Put \((a,b,t)=(r^2+r+3,r+2,r)\).  The signed degree-sum equation is equivalent
to \(s(v)=-d(v)^2+ad(v)-1\) on \(V_+\), and to
\(s(v)=-d(v)^2+ad(v)-(r+1)\) on \(V_-\).  It remains to verify these coordinate
identities, with all degrees and neighbor-degree sums taken in \(F_r(H)\).  If
\(w\in X\),
then
\(s(w)=r^2(r+1)+(r^2+r+1)\); if \(w\in Y\), then
\(s(w)=(r+1)+r(r^2+r+1)+r(r-1)\).  In both cases \(d(w)=r^2+1\) and
\(s(w)=-(r^2+1)^2+a(r^2+1)-1\), as required for vertices of \(W_+\subseteq
V_+\).

For \(u\in W_-\), \(d(u)=r+1\) and
\(s(u)=(r+1)(r^2+1)=-(r+1)^2+a(r+1)-(r+1)\).  Each center in
\(V(F_r(H))\setminus V(H)\) has degree \(r^2+r+1\) and neighbor-degree sum
\((r^2+1)+(r^2+r)\), which equals
\(-(r^2+r+1)^2+a(r^2+r+1)-(r+1)\).  A leaf in \(V_+\) has
\(d=1\) and \(s=r^2+r+1=a-2=-d^2+ad-1\), while a leaf in \(V_-\) has
\(d=1\) and \(s=r^2+1=a-r-2=-d^2+ad-(r+1)\).  Hence, in vector form,
\[
        Q^2\one-(r^2+r+3)Q\one+(r+2)\one=r\sgn .
\]
Taking the inner product with \(\sgn\) gives
\[
        r|V(F_r(H))|=(r+2)(|V_+|-|V_-|),
\]
because \(Q\sgn=0\) and \(\langle\sgn,\sgn\rangle=|V(F_r(H))|\).  Hence
\(|V_+|-|V_-|=r|V(F_r(H))|/(r+2)>0\).  Thus the bipartition inherited from the
construction is already ordered with \(V_+\) as the larger class, and the
displayed vector identity is the required signed degree-sum equation.
\end{proof}

\begin{lemma} \label{lem:skeleton}
For any integers \(d_+\ge4\) and \(d_-\ge3\), and any nonnegative integer
\(k\), there is a bipartite graph \(H\) of cyclomatic number \(k\), with
bipartition \(W_+\sqcup W_-\), such that all vertices in \(W_-\) have degree
\(d_-\), the degrees in \(W_+\) lie in \(\{1,d_+\}\), and \(W_+\) contains a
leaf.
\end{lemma}

\begin{proof}
Let \(L_k=P_{k+1}\square K_2\) be a ladder with vertices
\(u_0,\ldots,u_k\) on one path and \(v_0,\ldots,v_k\) on the other.  Its
bipartition \(A \sqcup B\) is given by
\[
        A =\{u_i:i\ \text{even}\}\cup\{v_i:i\ \text{odd}\},
        \qquad
        B =\{u_i:i\ \text{odd}\}\cup\{v_i:i\ \text{even}\}.
\]
Then \(|V(L_k)|=2(k+1)\) and \(|E(L_k)|=3k+1\), so the cyclomatic number
of \(L_k\) is \(k\).  Figure~\ref{fig:ladder} illustrates the ladder.

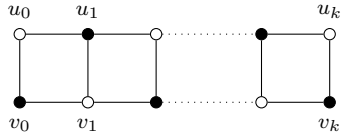
\begin{figure}[!ht]
\centering
\begin{tikzpicture}[scale=.82,
    pnode/.style={circle,draw,fill=white,inner sep=1.5pt},
    nnode/.style={circle,draw,fill=black,inner sep=1.5pt},
    every node/.style={font=\scriptsize}]
  \node[pnode] (u0) at (0,.55) {};
  \node[nnode] (v0) at (0,-.55) {};
  \node[nnode] (u1) at (1.1,.55) {};
  \node[pnode] (v1) at (1.1,-.55) {};
  \node[pnode] (u2) at (2.2,.55) {};
  \node[nnode] (v2) at (2.2,-.55) {};
  \node[nnode] (u3) at (3.9,.55) {};
  \node[pnode] (v3) at (3.9,-.55) {};
  \node[pnode] (u4) at (5.0,.55) {};
  \node[nnode] (v4) at (5.0,-.55) {};
  \foreach \i in {0, 1, 3}{
    \pgfmathtruncatemacro{\j}{\i+1}
    \draw (u\i)--(u\j);
    \draw (v\i)--(v\j);
  }
  \draw[dotted] (u2)--(u3);
  \draw[dotted] (v2)--(v3);
  \foreach \i in {0, 1, 2, 3, 4}{\draw (u\i)--(v\i);}
  \node[above=2pt] at (u0) {$u_0$};
  \node[above=2pt] at (u1) {$u_1$};
  \node[above=2pt] at (u4) {$u_k$};
  \node[below=2pt] at (v0) {$v_0$};
  \node[below=2pt] at (v1) {$v_1$};
  \node[below=2pt] at (v4) {$v_k$};
\end{tikzpicture}
\caption{The ladder \(L_k\) and its bipartition.}
\label{fig:ladder}
\end{figure}

Since \(\Delta(L_k)\le3\), the assumptions \(d_+\ge4\) and \(d_-\ge3\) give
\(d_+-d_{L_k}(w)>0\) for \(w\in A\) and
\(d_--d_{L_k}(u)\ge0\) for \(u\in B\).  Construct \(H\) from \(L_k\), with
\(A\subseteq W_+\) and \(B\subseteq W_-\), by making the following disjoint
attachments.  For each \(w\in A\), identify \(w\) with one leaf in each of
\(d_+-d_{L_k}(w)\) new copies of \(K_{1,d_-}\).  For each \(u\in B\), attach
\(d_--d_{L_k}(u)\) pendant edges at \(u\).
The centers of the attached stars lie in \(W_-\), and all their new leaves lie
in \(W_+\); the new endpoints of the pendant edges also lie in \(W_+\).  Since
only pendant trees are attached, the cyclomatic number of \(H\) remains \(k\).
By construction, every vertex in \(W_-\) has degree \(d_-\), and every vertex
in \(W_+\) has degree \(1\) or \(d_+\).  Since
\(A\ne\emptyset\) and \(d_-\ge3\), at least one attached star has a remaining
leaf in \(W_+\).
\end{proof}

\begin{theorem} \label{thm:infinite-family}
For every \(k\ge0\) and \(r\ge2\), the class
\(\calG_k(r^2+r+3,r+2,r)\) contains graphs of arbitrarily large diameter;
in particular, it contains infinitely many pairwise nonisomorphic graphs.  Each
has \(Q\)-main polynomial
\[
        x\bigl(x^2-(r^2+r+3)x+(r+2)\bigr).
\]
\end{theorem}

\begin{proof}
Apply Lemma~\ref{lem:skeleton} to the prescribed \(k\), with \(d_+=r^2\) and
\(d_-=r+1\).  Let \(H\), with bipartition \(W_+\sqcup W_-\), be the resulting
graph, and fix a leaf in \(W_+\).

Let \(T_0\) be the one-vertex rooted tree.  For \(m\ge1\), take a copy of
\(K_{1,r^2-1}\), with center chosen as the root.  For each leaf of this star,
take \(r\) new copies of \(T_{m-1}\), and add an edge from the leaf to the root
of each copy.  Thus, if \(n_m=|V(T_m)|\), then \(n_0=1\) and
\(n_m=r^2+r(r^2-1)n_{m-1}\) for \(m\ge1\).
Let \(H_m\) be obtained
by identifying the root of \(T_m\) with the fixed leaf of \(H\), assigning the
vertices of \(T_m\) at even distance from the root to \(W_+\) and those at odd
distance to \(W_-\).  Put \(G_m=F_r(H_m)\).
We write \(G_{r,k,m}\) for \(G_m\) when the dependence on \(r\) and \(k\) is to
be made explicit.

For every \(m\ge0\), the graph \(H_m\) has cyclomatic number \(k\), all vertices
in \(W_-\) have degree \(r+1\), and all vertices in \(W_+\) have degree \(1\) or
\(r^2\).  Lemma~\ref{lem:extension} therefore gives that \(G_m\) has
cyclomatic number \(k\) and satisfies the
\((r^2+r+3,r+2,r)\)-signed degree-sum equation.
The \(V_+\)-class contains vertices of degrees \(r^2+1\) and \(1\), so
\(G_m\) is not semiregular.  Since
\((r^2+r+3)^2>4(r+2)\), Theorem~\ref{thm:three-main-criterion} gives exactly
three \(Q\)-main eigenvalues, including the zero eigenvalue, with \(Q\)-main
polynomial
\[
        x\bigl(x^2-(r^2+r+3)x+(r+2)\bigr).
\]

For \(m\ge1\), choose vertices \(p,q\in V(T_m)\), each at distance \(2m\) from
the root, in different components of \(T_m\) after deleting the root.  Then
\(\operatorname{dist}_{H_m}(p,q)=4m\).  Construction~\ref{con:F} adds vertices
\(p'\) and \(q'\) at distance \(2\) from \(p\) and \(q\), respectively, and its
attached trees meet \(H_m\) only at their attachment vertices.  Hence
\(\operatorname{diam}(G_m)\ge4m+4\), proving the asserted unboundedness and
therefore the infinitude up to isomorphism.
\end{proof}

The cases \(k=0, 1, 2\) with \(r=2\) and \(m=0\) are displayed in
Figure~\ref{fig:r2-m0-examples}.

\begin{figure}
\centering
\includegraphics[width=\textwidth]{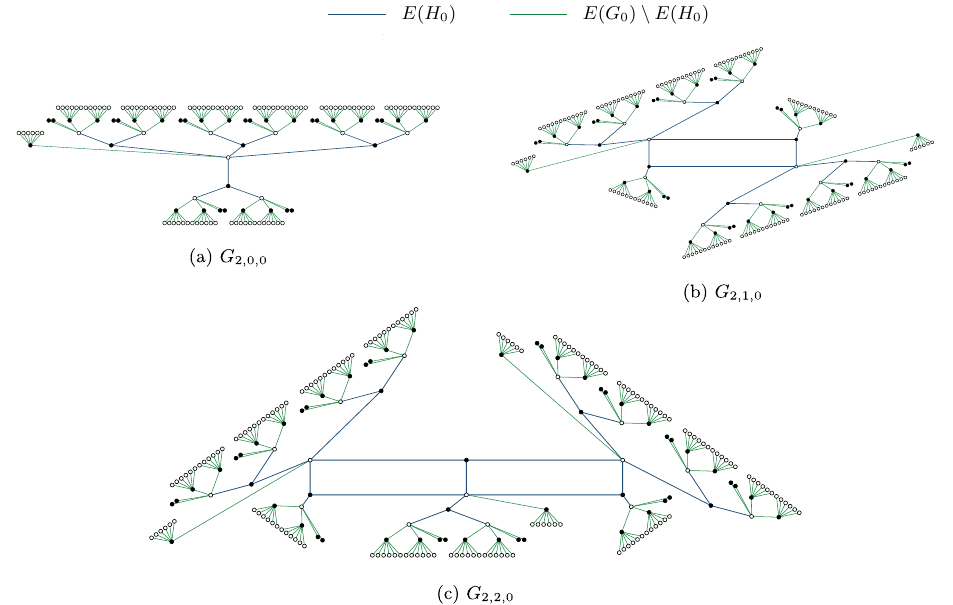}
\caption{Examples for \(r=2\), \(m=0\), and \(k=0,1,2\).}
\label{fig:r2-m0-examples}
\end{figure}

\subsection*{Consequences for previous classifications}

For graphs with exactly three \(Q\)-main eigenvalues, one of which is zero,
the earlier classifications may be summarized as follows.  For trees,
\cite[Theorem~4.3]{JavarsinehFathTabar2017Filomat} asserts that every such
graph is either obtained from a path on five vertices by attaching the same
number of additional leaves to the two vertices adjacent to its endvertices,
or is obtained from a path on seven vertices by attaching one additional leaf
to each of those two vertices; see
Figure~\ref{fig:claimed-tree-classification}.

For unicyclic graphs, \cite[Theorem~3.1]{JavarsinehFathTabar2017AMC} asserts
that every such graph is obtained from an even cycle by attaching the same
positive number of leaves to alternate cycle vertices.  For bicyclic graphs,
\cite[Theorems~4.1 and~4.3]{JavarsinehFathTabar2017AMC} asserts that the only
pendant-free examples are two \(4\)-cycles sharing one vertex, two disjoint
\(4\)-cycles joined by a path of length \(4\), and three internally disjoint
paths of length \(4\) with common endpoints, and that no example has a
nonempty pendant part; see Figure~\ref{fig:claimed-amc-classification}.

Theorem~\ref{thm:infinite-family} contradicts all three assertions.  For
\(k=0\), it gives trees of arbitrarily large diameter, whereas the claimed
trees have diameter four or six.  For \(k=1\), it gives vertices arbitrarily
far from the unique cycle, whereas every vertex off the cycle in the claimed
family is pendant.  For \(k=2\), it gives bicyclic graphs with pendant
vertices, whereas the claimed list contains only the three pendant-free graphs
above.  Hence the three cited classification conclusions are not exhaustive.

\begin{figure}[!ht]
\centering
\vspace{4pt}
\begin{tikzpicture}[
    scale=.82,
    vertex/.style={circle,draw,fill=black,inner sep=1.5pt},
    every node/.style={font=\scriptsize}]
  \begin{scope}[xshift=-4.2cm]
    \foreach \x in {-2,-1,0,1,2}
      \node[vertex] (t\x) at (\x,0) {};
    \draw (t-2)--(t-1)--(t0)--(t1)--(t2);
    \node[vertex] (t1l) at (-1.65,-.82) {};
    \node[vertex] (t1r) at (-.35,-.82) {};
    \draw (t-1)--(t1l) (t-1)--(t1r);
    \node at (-1,-.72) {$\cdots$};
    \node[vertex] (t3l) at (.35,-.82) {};
    \node[vertex] (t3r) at (1.65,-.82) {};
    \draw (t1)--(t3l) (t1)--(t3r);
    \node at (1,-.72) {$\cdots$};
    \node at (0,-1.35) {(a)};
  \end{scope}
  \begin{scope}[xshift=4.1cm]
    \foreach \x in {-3,-2,-1,0,1,2,3}
      \node[vertex] (s\x) at (\x,0) {};
    \draw (s-3)--(s-2)--(s-1)--(s0)--(s1)--(s2)--(s3);
    \node[vertex] (sl) at (-2,-.82) {};
    \node[vertex] (sr) at (2,-.82) {};
    \draw (s-2)--(sl) (s2)--(sr);
    \node at (0,-1.35) {(b)};
  \end{scope}
\end{tikzpicture}
\caption{The trees listed in
\protect\cite[Theorem~4.3]{JavarsinehFathTabar2017Filomat}.}
\label{fig:claimed-tree-classification}
\end{figure}

\begin{figure}[!ht]
\centering
\begin{tikzpicture}[
    scale=.82,
    vertex/.style={circle,draw,fill=black,inner sep=1.5pt},
    every node/.style={font=\scriptsize}]
  \node[vertex] (cl)  at (-2.6,0) {};
  \node[vertex] (cul) at (-2.0,.9) {};
  \node[vertex] (cta) at (-1.1,.9) {};
  \node[vertex] (ctb) at (1.1,.9) {};
  \node[vertex] (cur) at (2.0,.9) {};
  \node[vertex] (cr)  at (2.6,0) {};
  \node[vertex] (clr) at (2.0,-.9) {};
  \node[vertex] (cba) at (1.1,-.9) {};
  \node[vertex] (cbb) at (-1.1,-.9) {};
  \node[vertex] (cll) at (-2.0,-.9) {};
  \draw (cl)--(cul)--(cta)--(-.32,.9);
  \draw (.32,.9)--(ctb)--(cur)--(cr)--(clr)--(cba)--(.32,-.9);
  \draw (-.32,-.9)--(cbb)--(cll)--(cl);
  \node[inner sep=0pt] at (0,.9) {$\scriptscriptstyle\cdots$};
  \node[inner sep=0pt] at (0,-.9) {$\scriptscriptstyle\cdots$};

  \foreach \x/\p in {-1.1/cta,1.1/ctb}{
    \node[vertex] (\p1) at ({\x-.38},1.78) {};
    \node[vertex] (\p2) at ({\x+.38},1.78) {};
    \draw (\p)--(\p1) (\p)--(\p2);
    \node[inner sep=0pt] at (\x,1.78) {$\cdots$};
  }
  \foreach \x/\p in {-1.1/cbb,1.1/cba}{
    \node[vertex] (\p1) at ({\x-.38},-1.78) {};
    \node[vertex] (\p2) at ({\x+.38},-1.78) {};
    \draw (\p)--(\p1) (\p)--(\p2);
    \node[inner sep=0pt] at (\x,-1.78) {$\cdots$};
  }
  \node[vertex] (cl1) at (-3.35,.38) {};
  \node[vertex] (cl2) at (-3.35,-.38) {};
  \draw (cl)--(cl1) (cl)--(cl2);
  \node[inner sep=0pt] at (-3.35,0) {$\vdots$};
  \node[vertex] (cr1) at (3.35,.38) {};
  \node[vertex] (cr2) at (3.35,-.38) {};
  \draw (cr)--(cr1) (cr)--(cr2);
  \node[inner sep=0pt] at (3.35,0) {$\vdots$};
  \node at (0,-2.18) {(a)};
\end{tikzpicture}

\smallskip

\begin{tikzpicture}[
    scale=.74,
    vertex/.style={circle,draw,fill=black,inner sep=1.5pt},
    every node/.style={font=\scriptsize}]
  \begin{scope}[xshift=-6.4cm]
    \node[vertex] (bc) at (0,0) {};
    \node[vertex] (blt) at (-1,.8) {};
    \node[vertex] (blm) at (-2,0) {};
    \node[vertex] (blb) at (-1,-.8) {};
    \node[vertex] (brt) at (1,.8) {};
    \node[vertex] (brm) at (2,0) {};
    \node[vertex] (brb) at (1,-.8) {};
    \draw (bc)--(blt)--(blm)--(blb)--(bc)--(brt)--(brm)--(brb)--(bc);
    \node at (0,-1.35) {(b)};
  \end{scope}
  \begin{scope}
    \node[vertex] (du) at (-2,0) {};
    \node[vertex] (dlt) at (-2.75,.75) {};
    \node[vertex] (dlm) at (-3.5,0) {};
    \node[vertex] (dlb) at (-2.75,-.75) {};
    \node[vertex] (dw1) at (-1,0) {};
    \node[vertex] (dw2) at (0,0) {};
    \node[vertex] (dw3) at (1,0) {};
    \node[vertex] (dz) at (2,0) {};
    \node[vertex] (drt) at (2.75,.75) {};
    \node[vertex] (drm) at (3.5,0) {};
    \node[vertex] (drb) at (2.75,-.75) {};
    \draw (du)--(dlt)--(dlm)--(dlb)--(du)--(dw1)--(dw2)--(dw3)--(dz)
      --(drt)--(drm)--(drb)--(dz);
    \node at (0,-1.35) {(c)};
  \end{scope}
  \begin{scope}[xshift=6.4cm]
    \node[vertex] (qu) at (-2,0) {};
    \node[vertex] (qz) at (2,0) {};
    \foreach \y/\s in {.8/t,0/m,-.8/b}{
      \node[vertex] (\s1) at (-1,\y) {};
      \node[vertex] (\s2) at (0,\y) {};
      \node[vertex] (\s3) at (1,\y) {};
      \draw (qu)--(\s1)--(\s2)--(\s3)--(qz);
    }
    \node at (0,-1.35) {(d)};
  \end{scope}
\end{tikzpicture}
\caption{The unicyclic and bicyclic graphs listed in
\protect\cite[Theorems~3.1 and~4.1]{JavarsinehFathTabar2017AMC}.}
\label{fig:claimed-amc-classification}
\end{figure}
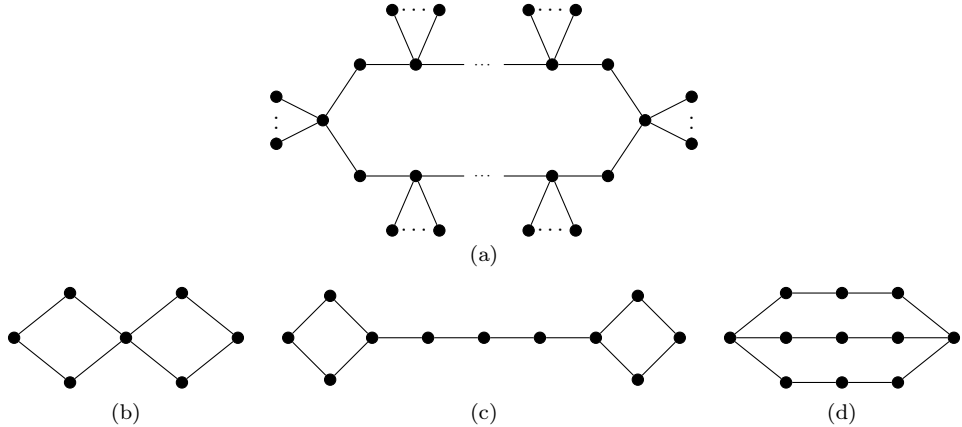

\FloatBarrier

\section{Further directions}\label{sec:open-problems}

The first remaining issues are already visible for exactly three \(Q\)-main
eigenvalues.  The calculation at the beginning of
Section~\ref{sec:families} singles out the natural triples
\[
        (t^2+t+q+2,t+2q,t) \qquad (q\ge1),
\]
obtained from \(b=t+2q\).  Theorem~\ref{thm:infinite-family} realizes the
first layer \(q=1\), \(t\ge2\), with arbitrary prescribed cyclomatic number;
the boundary line \(t=1\) is realized for \(q=1\) and for all \(q\ge3\) by the
constructions above.  The remaining questions therefore begin with the case
\(t=1\), \(q=2\), and with higher layers for \(t\ge2\).

For more than three \(Q\)-main eigenvalues,
Theorem~\ref{thm:zero-main-general-criterion} still applies, but the equation
\(p(Q)\one=(-1)^{\ell-1}t\sgn\) involves \(Q^j\one\) with \(j\ge3\).  Already
for \(\ell=4\), if \(p(x)=x^3-ax^2+bx-c\), the vector \(Q^3\one\) has coordinate
\[
        (Q^3\one)_v
        =2d(v)^3+2d(v)s(v)+2\sum_{u\sim v}\bigl(d(u)^2+s(u)\bigr),
\]
so the resulting condition is no longer determined only by the degree and
2-degree of \(v\).  This suggests the following problems.

\begin{enumerate}[label=(\arabic*),leftmargin=*]
\item \textbf{Higher layers for \(\ell=3\).}  Determine whether the triples
\((t^2+t+q+2,t+2q,t)\) with \(t\ge2\) and \(q\ge2\) are realizable.  If so, can
the cyclomatic number be prescribed?  The exceptional boundary case
\((a,b,t)=(6,5,1)\) also remains to be understood.

\item \textbf{Higher-degree main polynomials.}  For \(\ell\ge4\), do there
exist infinite families with prescribed cyclomatic number and exactly \(\ell\)
\(Q\)-main eigenvalues, including the zero eigenvalue?  Can one design
constructions that enforce \(p(Q)\one=(-1)^{\ell-1}t\sgn\) for prescribed
polynomials \(p\) of degree \(\ell-1\)?

\item \textbf{Small cyclomatic number.}  Give complete classifications of the
trees, unicyclic graphs, and bicyclic graphs with exactly three \(Q\)-main
eigenvalues, including the zero eigenvalue.

\item \textbf{The zero-free case.}  Graphs with exactly three \(Q\)-main
eigenvalues, none of which is zero, are not covered by the bipartite reduction
used here.  Is there an analogous vertexwise criterion in this case?
\end{enumerate}

\section*{Data Availability Statement}
\phantomsection
\label{sec:data-availability}
The SageMath verification notebook, construction scripts, and sample output
files used in this paper are publicly available in a
\href{https://github.com/ditf015/3-Q-main-eigenvalues-ce}{GitHub repository}.
The corresponding archived release, version v1.1.0, is available on Zenodo:
\url{https://doi.org/10.5281/zenodo.21432139}.

\bibliographystyle{elsarticle-num}
\bibliography{v1_refs}

@article{Cvetkovic2007,
  author  = {Cvetkovi{\'c}, Drago{\v{s}} and Rowlinson, Peter and Simi{\'c}, Slobodan K.},
  title   = {Signless {Laplacians} of finite graphs},
  journal = {Linear Algebra and its Applications},
  volume  = {423},
  number  = {1},
  year    = {2007},
  pages   = {155--171},
  doi     = {10.1016/j.laa.2007.01.009}
}

@book{CvetkovicBook,
  author    = {Cvetkovi{\'c}, Drago{\v{s}} and Rowlinson, Peter and Simi{\'c}, Slobodan},
  title     = {An Introduction to the Theory of Graph Spectra},
  series    = {London Mathematical Society Student Texts},
  volume    = {75},
  publisher = {Cambridge University Press},
  address   = {Cambridge},
  year      = {2010},
  doi       = {10.1017/CBO9780511801518}
}

@article{DengHuang2013,
  author  = {Deng, Hanyuan and Huang, He},
  title   = {On the main signless {Laplacian} eigenvalues of a graph},
  journal = {Electronic Journal of Linear Algebra},
  volume  = {26},
  year    = {2013},
  pages   = {381--393},
  doi     = {10.13001/1081-3810.1659}
}

@article{FengYu2009,
  author  = {Feng, Lihua and Yu, Guihai},
  title   = {On three conjectures involving the signless {Laplacian} spectral radius of graphs},
  journal = {Publ. Inst. Math. (Beograd) (N.S.)},
  volume  = {85},
  number  = {99},
  year    = {2009},
  pages   = {35--38},
  doi     = {10.2298/PIM0999035F}
}

@article{ChenHuang2013,
  author  = {Chen, Lin and Huang, Qiongxiang},
  title   = {Trees, unicyclic graphs and bicyclic graphs with exactly two {$Q$}-main eigenvalues},
  journal = {Acta Mathematica Sinica, English Series},
  volume  = {29},
  number  = {11},
  year    = {2013},
  pages   = {2193--2208},
  doi     = {10.1007/s10114-013-1629-y}
}

@misc{LiYang2013Tricyclic,
  author        = {Li, Shuchao and Yang, Xue},
  title         = {Characterization of tricyclic graphs with exactly two {$Q$}-main eigenvalues},
  note          = {arXiv:1304.3524},
  year          = {2013},
  eprint        = {1304.3524},
  archivePrefix = {arXiv},
  primaryClass  = {math.CO},
  doi           = {10.48550/arXiv.1304.3524}
}

@article{JavarsinehFathTabar2017Filomat,
  author  = {Javarsineh, Mehrnoosh and Fath-Tabar, Gholam Hossein},
  title   = {On graphs with exactly three {$Q$}-main eigenvalues},
  journal = {Filomat},
  volume  = {31},
  number  = {6},
  year    = {2017},
  pages   = {1803--1812},
  doi     = {10.2298/FIL1706803J}
}

@article{JavarsinehFathTabar2017AMC,
  author  = {Javarsineh, Mehrnoosh and Fath-Tabar, Gholam Hossein},
  title   = {Unicyclic and bicyclic graphs with exactly three {$Q$}-main eigenvalues},
  journal = {Applied Mathematics and Computation},
  volume  = {315},
  year    = {2017},
  pages   = {603--614},
  doi     = {10.1016/j.amc.2017.06.033}
}

@article{Rowlinson2007,
  author  = {Rowlinson, Peter},
  title   = {The main eigenvalues of a graph: A survey},
  journal = {Applicable Analysis and Discrete Mathematics},
  volume  = {1},
  number  = {2},
  year    = {2007},
  pages   = {445--471},
  doi     = {10.2298/AADM0702445R}
}

@article{Cvetkovic1970,
  author  = {Cvetkovi{\'c}, Drago{\v{s}} M.},
  title   = {The generating function for variations with restrictions and paths of the graph and self-complementary graphs},
  journal = {Publikacije Elektrotehni{\v{c}}kog Fakulteta. Serija Matematika i Fizika},
  number  = {320/328},
  year    = {1970},
  pages   = {27--34}
}

@article{CvetkovicSimic2009,
  author  = {Cvetkovi{\'c}, Drago{\v{s}} and Simi{\'c}, Slobodan K.},
  title   = {Towards a spectral theory of graphs based on the signless {Laplacian}, {I}},
  journal = {Publ. Inst. Math. (Beograd) (N.S.)},
  volume  = {85},
  number  = {99},
  year    = {2009},
  pages   = {19--33},
  doi     = {10.2298/PIM0999019C}
}

@article{Hagos2002,
  author  = {Hagos, Elias M.},
  title   = {Some results on graph spectra},
  journal = {Linear Algebra and its Applications},
  volume  = {356},
  year    = {2002},
  pages   = {103--111},
  doi     = {10.1016/S0024-3795(02)00324-5}
}

@article{VinagreTrevisanBolckauChimelli2020,
  author  = {Vinagre, Cybele T. M. and Trevisan, Vilmar and Bolckau, Johann and Chimelli, Rodrigo},
  title   = {Characterizing threshold graphs with {$k$} main signless {Laplacian} eigenvalues},
  journal = {Linear Algebra and its Applications},
  volume  = {602},
  year    = {2020},
  pages   = {33--45},
  doi     = {10.1016/j.laa.2020.04.022}
}

@article{JonesTrevisanVinagre2025,
  author  = {Jones, {\'A}tila and Trevisan, Vilmar and Vinagre, Cybele T. M.},
  title   = {Characterization of quasi-threshold graphs with two main {$Q$}-eigenvalues},
  journal = {Linear Algebra and its Applications},
  volume  = {711},
  year    = {2025},
  pages   = {68--83},
  doi     = {10.1016/j.laa.2025.02.009}
}

@book{Bapat2010,
  author    = {Bapat, Ravindra B.},
  title     = {Graphs and Matrices},
  series    = {Universitext},
  publisher = {Springer},
  address   = {London},
  year      = {2010},
  doi       = {10.1007/978-1-84882-981-7}
}

@article{GutknechtSchmelzer2009,
  author  = {Gutknecht, Martin H. and Schmelzer, Thomas},
  title   = {The block grade of a block {Krylov} space},
  journal = {Linear Algebra and its Applications},
  volume  = {430},
  number  = {1},
  year    = {2009},
  pages   = {174--185},
  doi     = {10.1016/j.laa.2008.07.008}
}

\end{document}